\documentclass[a4paper]{article}
\usepackage[usenames,dvips]{color} 
\usepackage{amsthm,amssymb,amsmath,bbm,enumerate,graphicx,epsf,stmaryrd}
\usepackage{hyperref}
\usepackage{tikz}
\usepackage{url}
\usetikzlibrary{arrows}
\usepackage{subcaption}
\usepackage[margin=1.3in]{geometry}

\usepackage{tkz-graph}
\usetikzlibrary{arrows}
\newcommand{\keywordname}{{\bf Keywords:}}
\newcommand{\keywords}[1]{\par\addvspace\baselineskip
	\noindent\keywordname\enspace\ignorespaces#1}

\newtheorem{theorem}{Theorem}
\newtheorem{lemma}[theorem]{Lemma}
\newtheorem{corollary}[theorem]{Corollary}
\newtheorem{proposition}[theorem]{Proposition}
 
\usepackage{setspace}
\newcommand{\exu}[2]{\mathbb{E}_{#1}\!\left[#2\right]}

\renewcommand{\Pr}[1]{\mathbb{P}\left[\,#1\,\right]}

\newcommand{\Ex}[1]{\mathbb{E} \left[\,#1\,\right]}
\newcommand{\Exu}[2]{\mathbb{E}_{#1} \left[\,#2\,\right]} 

\newcommand{\Pruc}[3]{\mathbb{P}_{#1}^{#2} \left[\,#3\,\right]} 
\newcommand{\Exuc}[3]{\mathbb{E}_{#1}^{#2} \left[\,#3\,\right]}

\title{The Cover Time of a (Multiple) Markov Chain with Rational Transition Probabilities is Rational}
 
\author{John Sylvester 
\\
  {\small School of Computing Science, University of Glasgow, Glasgow, UK }\\
{ \small \tt john.sylvester@glasgow.ac.uk }}

\date{}
\begin{document}
\maketitle
\begin{abstract}The cover time of a Markov chain on a finite state space is the expected time until all states are visited. We show that if the cover time of a discrete-time Markov chain with rational transitions probabilities is bounded, then it is a rational number. The result is proved by relating the cover time of the original chain to the hitting time of a set in another higher dimensional chain. We prove this result in a more general setting where $k\geq 1 $ independent copies of a Markov chain are run simultaneously on the same state space.
\end{abstract}
\section{Introduction and Results}
 Let $(X_t)_{t\geq 0}$ be a discrete-time Markov chain with transition matrix   $\mathbf P$ on a state space $\Omega$, see \cite{aldousfill,LevinPeres} for background.  We say a chain is \textit{rational} if all its transition probabilities are rational numbers, i.e.\ $\mathbf{P}(x,y) \in\mathbb{Q}$ for all $x,y \in \Omega$. 
The stopping time $\tau_{\mathsf{cov}}$ is the first time all states are visited, that is
\[ \tau_{\mathsf{cov}}: = \inf \left\{t\geq 0 : \bigcup_{k=0}^t\{X_k \} = \Omega \right\} . \] For $x \in \Omega$, let $\Exu{x}{\tau_{\mathsf{cov}}}=\Ex{\tau_{\mathsf{cov}}\mid X_0=x}$ be the \textit{cover time} from $x$, that is, the expected time for the chain to visit all states when started from $x\in \Omega$. 

Along with mixing and hitting times, the cover time is one of the most natural and well studied stopping times for a Markov chain and has found applications in the analysis of algorithms, see for example \cite{MultiAlon}, \cite[Ch.\ 6.8]{aldousfill} and \cite[Ch.\ 11]{LevinPeres}. It is clear that the stopping time $\tau_{\mathsf{cov}}$ is a natural number, however it is not so clear whether the cover time $\Exu{x}{\tau_{\mathsf{cov}}}$ is rational, even if the transition probabilities are rational. Our main result shows that, under some natural assumptions, the cover time of a rational Markov chain is rational.

\begin{theorem}\label{covrat}Let $(X_t)_{t\geq 0}$ be a discrete-time rational Markov chain on a finite state space $\Omega$. Then, for any $x \in \Omega$ such that $\exu{x}{\tau_{\mathsf{cov}}}<\infty$, we have $\exu{x}{\tau_{\mathsf{cov}}}\in \mathbb{Q}$.
\end{theorem}

The assumption that $\Omega$ is finite is necessary to ensure the cover time is bounded. Recall that a Markov chain is \textit{irreducible} if for every $x,y\in \Omega$ there exists some $t\geq 0$ such that $\mathbf{P}^t(x,y)>0$, where $\mathbf{P}^t(x,y)$ denotes the probability a chain started at $x$ is at state $y$ after $t\geq 1$ steps. Theorem \ref{covrat} does not require irreducibility, just that the cover time from the given start vertex is bounded. An example of a non-irreducible Markov chain to which we can apply Theorem \ref{covrat} is given in Figure \ref{fig:example}. In this example the cover time from $x$ is bounded however, the cover time from any other vertex is unbounded/undefined, as if a walk starts from any other vertex, then $x$ (and possibly also the vertex immediately right of $x$) cannot be reached. 

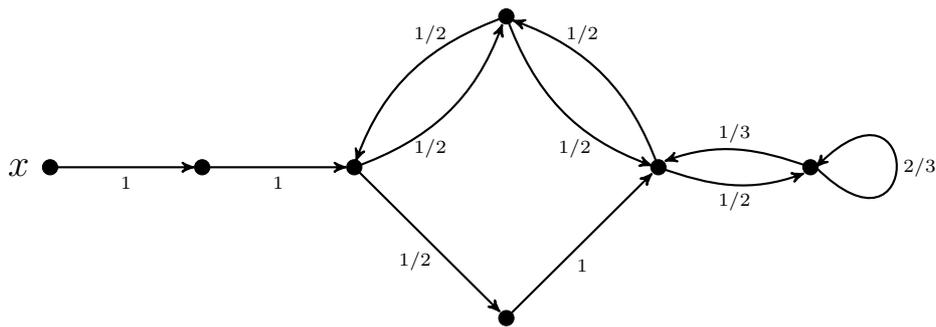
\begin{figure}
	
	\begin{center}
		
		\begin{tikzpicture}[label/.style={thick,circle}]
		\usetikzlibrary{arrows.meta}
		\usetikzlibrary{decorations.markings}
		\usetikzlibrary{decorations.pathreplacing}
		\tikzset{->-/.style={decoration={
					markings,
					mark=at position .5 with {Stealth[length=4mm]}},postaction={decorate}},>=stealth'}
	 
	    \draw[fill] (0,0) circle (.1);
		\draw (-.15,0) node[anchor=east]{{ \Large  $x$}};
		\draw[->,thick] (0 ,0) to (1.9,0);
		\draw (1,0) node[anchor=north]{{\scriptsize$1$}};
		
		\draw[fill] (2,0) circle (.1);
			\draw[->,thick] (2,0) to   (3.9,0);
		\draw (3,0) node[anchor=north]{{\scriptsize$  1$}};
		\draw[fill] (4,0) circle (.1);
			\draw[->,thick] (4,0) to[out=20,in=250 ]  (5.95,1.90);
		\draw (5,.5) node[anchor=north]{{\scriptsize$  1/2$}};
			\draw[->,thick] (4,0) to   (5.92,-1.92);
		\draw (4.8,-1) node[anchor=north]{{\scriptsize$  1/2$}};
		
		\draw[fill] (6,2) circle (.1);
				\draw[->,thick] (6,2) to[out=200,in=70 ]   (4.02,.08);
		\draw (5,2) node[anchor=north]{{\scriptsize$  1/2$}};
			\draw[->,thick] (8,0) to[out=110,in=-20 ]  (6.08,1.95);
		\draw (7,2) node[anchor=north]{{\scriptsize$  1/2$}};
		
		\draw[fill] (6,-2) circle (.1);
			\draw[->,thick] (6,-2) to   (7.92,-.08);
		\draw (7,-1.1) node[anchor=north]{{\scriptsize$  1$}};
	
		\draw[fill] (8,0) circle (.1);
		\draw[->,thick] (8,0) to[out=-20,in=200 ]  (9.92,-.08);
	\draw (9,-.2) node[anchor=north]{{\scriptsize$  1/2$}};
		\draw[<-,thick] (7.92,0.02) to[out=160,in=290 ]  (6,2);
	\draw (6.9,.5) node[anchor=north]{{\scriptsize$  1/2$}};
		\draw[fill] (10,0) circle (.1);
				\draw[->,thick] (10,0) to[out=160,in=20 ]  (8.08,.08);
		\draw (9,.2) node[anchor=south]{{\scriptsize$  1/3$}};
		\Loop[dist=2cm,dir=EA,style={thick,->}](10.07,0)
		\draw (11.1,0) node[anchor=west]{{\scriptsize$  2/3$}};

		\end{tikzpicture}
	\end{center}
	\caption{Example of a non-irreducible Markov chain on seven states where the cover time from $x$ is finite and from any other vertex the cover time is unbounded/undefined. }\label{fig:example}
\end{figure}

For a concrete example of why rational transition probabilities are necessary in Theorem \ref{covrat}, if one fixes any real number $r\geq 1$ then the two state chain with transition matrix given by
\begin{equation}\label{eq:matrix}\mathbf{P}= \begin{pmatrix}
1-1/r & 1/r\\
1/r & 1-1/r
\end{pmatrix}, \end{equation}has cover time $r$. 
It is well known, see for example \cite[Lemma 1.13]{LevinPeres}, that the cover time of finite irreducible Markov chain from any start vertex is bounded. 
This fact, and restricting the example given by \eqref{eq:matrix} to $r\in\mathbb{Q}$, implies the following corollary to Theorem \ref{covrat}. 
\begin{corollary} The set of cover times attainable by finite discrete-time irreducible rational Markov chains is $(\mathbb{Q}\cap [1,\infty) )\cup \{0 \}. $ 
\end{corollary}  
 
We now introduce multiple Markov chains, which have been studied for their applications to parallelising algorithms driven by random walks, see \cite{MultiAlon} and subsequent papers citing it. For any $ k \geq 1$, let $\mathbf{X}_t=\bigl(X_t^{(1)},\dots, X_t^{(k)} \bigr)$ be the $k$-\textit{multiple} of a Markov chain $\mathbf{P}$ where each $X_{t}^{(i)}$ is an independent copy of the chain  $\mathbf{P}$ run simultaneously on the same state space $\Omega$. The $k$-multiple of $\mathbf{P}$ is itself a Markov chain (with transition matrix $\mathbf{K}$) on $\Omega^k$ with transition probabilities 
\[\mathbf{K}(\mathbf{x},\mathbf{y}) = \prod_{j=1}^k\mathbf{P}(x^{(j)},y^{(j)}),\qquad\qquad\text{for all }\mathbf{x},\mathbf{y}\in\Omega^k.\] As before, we denote the conditional expectation $\Exu{(x^{(1)}, \dots, x^{(k)})}{\cdot } := \Ex{\cdot  \;\big| \mathbf{X}_0= (x^{(1)}, \dots, x^{(k)})}$,  where $X_{0}^{(i)} =x^{(i)}\in \Omega$ is the start state of the $i$\textsuperscript{th} walk for each $1\leq i\leq k$. We let the stopping time $\tau_{\mathsf{cov}}^{(k)}=\inf\{t : \bigcup_{i=0}^t\{X_i^{(1)}, \dots , X_i^{(k)} \} = \Omega  \}$ be the first time every state in $\Omega$ (not $\Omega^k$) has been visited by some walk
$X_t^{(i)}$. We then let $\Exu{\mathbf{x}}{\tau_{\mathsf{cov}}^{(k)}}$ denote the $k$-walk stopping time from $\mathbf{x}\in \Omega^k$. Note that this is not simply the cover time of the chain $\mathbf{K}$. The multiple walk cover time can have subtle dependences on $k$ and the host underlying Markov chain, see \cite{MultiAlon}.

We show that Theorem \ref{covrat} also holds in the more general setting of $k$-multiple Markov chains.    
\begin{theorem}\label{multicovrat}Let $k\geq 1$ and $(\mathbf{X}_t)_{t\geq 0}$ be the $k$-multiple of a discrete-time rational Markov chain on a finite state space $\Omega$. Then for any $\mathbf{x}\in \Omega^k$ such that $\exu{\mathbf{x}}{\tau_{\mathsf{cov}}^{(k)}}<\infty$ we have $\exu{\mathbf{x}}{\tau_{\mathsf{cov}}^{(k)}}\in \mathbb{Q}$.
\end{theorem}
Theorem \ref{covrat} is the special case $k=1$ of Theorem \ref{multicovrat}, thus it suffices to prove Theorem \ref{multicovrat}.
\section{Proofs} 

In this section we shall prove Theorem \ref{multicovrat}. The first part of the proof (covered in Section \ref{sec:rat}) is to show the expected time to first visit any set of states (hitting time) in a rational Markov chain is rational. Then, in Section \ref{sec:couple}, we show for any $k\geq1$ and $\mathbf{P}$, the multiple walk with transition matrix $\mathbf{P}$ can be coupled with a higher dimensional Markov chain $\mathbf{Q}$ on a state space $V$ where $|V|\leq |\Omega|^k\cdot 2^{|\Omega|}$. The coupling shows that the first time all states in $\Omega$ have been visited by at least one of the $k$ walks has the same distribution as the first visit time a specific set $C\subset V$ is visited in $\mathbf{Q}$.

\subsection{Rationality of Hitting Times}\label{sec:rat}

For $S\subseteq \Omega$, a subset of the state space of a Markov chain $\mathbf{P}$, let the stopping time\[\tau_{S}:=\inf\left\{t \geq 0 : X_t \in S  \right\},\]be the first time $S$ is visited. If $S=\{s\}$ is a singleton set we abuse notation slightly by taking $\tau_s$ to mean $\tau_{\{s\}}$. For $x\in \Omega$, let $\Exu{x}{\tau_S}$ be the expected \textit{hitting time} of $S \subseteq \Omega$ for a chain started from $x$. The next result is the hitting time analogue of Theorem \ref{covrat}.

\begin{proposition}\label{hitrat+}Let $\mathbf{P}$ be a discrete-time rational Markov chain on a finite state space $\Omega$. For a non-empty set $S\subseteq \Omega$ let $B(S) = \left\{x \in \Omega \,: \, \Exu{x}{\tau_S} < \infty \right\} $.  Then for any $S\subseteq \Omega$ and $x \in B(S)$ we have $\exu{x}{ \tau_S }\in \mathbb{Q}.$
\end{proposition}
Observe that if $\mathbf{P}$ is irreducible then $B(S)=\Omega$ for any $S\subseteq \Omega$ by \cite[Lemma 1.13]{LevinPeres}. Before proving Proposition \ref{hitrat+} we give some definitions and prove an elementary lemma. 

For a field $\mathbf{F}$ and integers $n,m \geq 1$ let $\mathbf{F}^n$ and $\mathbf{F}^{m\times n}$ denote the set of $n$-dimensional vectors and $m\!\times\! n$-dimension matrices respectively. Let $\mathbf{I}_n$ denote the $n\!\times\! n$ identity matrix.

\begin{lemma}\label{Ratsol} Let $\mathbf{A}\in \mathbb{Q}^{n\times n}$ be non-singular and $\mathbf{b}\in \mathbb{Q}^n$. Then there exists a unique vector $\mathbf{x}\in \mathbb{Q}^n$ such that $\mathbf{A}\mathbf{x}=\mathbf{b}$.
\end{lemma}
\begin{proof}Since $\mathbf{A} $ is non-singular there exists a unique solution $\mathbf{x}\in \mathbb{R}^n$ to the linear system given by $\mathbf{A}\mathbf{x}=\mathbf{b}$. Also, again since $\mathbf{A} $ is non-singular, we can compute $\mathbf{A}^{-1}$ by Gaussian elimination. Since all entries of $\mathbf{A}$ are rational, all multiplications preformed during the Gaussian elimination will be rational. Thus, as there are only finitely many row additions and multiplications, $\mathbf{A}^{-1}\in \mathbb{Q}^{n\times n}$. Since $\mathbf{b}\in \mathbb{Q}^{n} $, we conclude that $\mathbf{x}  =\mathbf{A}^{-1}\mathbf{b}\in \mathbb{Q}^{n}$.
\end{proof}
  We now use this lemma to prove Proposition \ref{hitrat+}.

\begin{proof}[Proof of Proposition \ref{hitrat+}]Observe that $B(S)\neq \emptyset$ since $S\subseteq B(S)$ and $\Exu{s}{\tau_S}=0$ for all $s\in S$. Let $b:=|B(S)|$. Now, each entry of the vector $\mathbf{h}:=\left( \exu{x}{\tau_S}\right)_{x \in B(S)}$ is bounded and $\mathbf{h}$ is a solution to the following set of linear equations
	\[\exu{x}{\tau_S}=\begin{cases}1+\sum_y \mathbf{P}(x,y)\cdot \exu{y}{\tau_S}&\quad\text{if }x\not\in S\\
	0&\quad\text{if }x\in S.\end{cases}\] This can be expressed as $ \mathbf A \mathbf h=\mathbf b$ where $\mathbf b\in \{0,1 \}^b$ and $\mathbf A :=\left(\mathbf{I}_b -\mathbf{M}\right) \in \mathbb{Q}^{b\times b}$ for $\mathbf{M}\in \mathbb{Q}^{b\times b}$ given by $\mathbf{M}(i,j) =\mathbf{P}(i,j) $ if $i,j\notin S  $ and $0$ otherwise. We shall show that 	
		\begin{enumerate} \item \label{itm:i} all rows $i$ satisfy $|\mathbf{A}(i,i)|\geq \sum_{j\neq i } |\mathbf{A}(i,j)|$, and
		\item\label{itm:ii} for each row $r_0$, there exists a finite sequence of rows $r_0,r_1 \dots, r_t$ such that $\mathbf{A}(r_{i-1},r_{i})\neq 0 $ for all $1\leq i\leq t$ and  $|\mathbf{A}(r_t,r_t)|> \sum_{j\neq r_t} |\mathbf{A}(r_t,j)|$.  
	\end{enumerate}
Observe that Condition \ref{itm:i} holds since $\mathbf{M}$ is a sub-matrix of $\mathbf{P}$. 

For Condition \ref{itm:ii}, note that for every row $s\in S$ we have $\sum_{j} \mathbf{M}(s,j) =0$. Thus $|\mathbf{A}(s,s)|> \sum_{j\neq s } |\mathbf{A}(s,j)|$ for any row $s \in S$. The fact that each row $r_0$ corresponds to a state in $B(S)$ implies that, for any row $r_0$, there exists some $r_t\in S$ and a sequence of states/rows $r_0,r_1 \dots, r_t$ such that $\mathbf{A}(r_{i-1},r_{i}) = -\mathbf{P}(r_{i-1},r_{i})\neq 0 $, thus Condition \ref{itm:ii} is satisfied.

 Since $\mathbf{A}$ satisfies \ref{itm:i} and \ref{itm:ii} it is weakly chained diagonally dominant, thus by \cite[Lemma 3.2]{AziFor} $\mathbf{A}$ is non-singular. Thus, by Lemma \ref{Ratsol}, $\mathbf{h}\in\mathbb{Q}^{b}$.
\end{proof}

\subsection{Encoding Cover Times as Hitting Times} \label{sec:couple}

Let $\mathbf{P}$ be a Markov chain on a state space $\Omega$ with transition matrix  $ \mathbf{P}=(\mathbf{P}(x,y))_{x,y\in \Omega}$ and $\mathcal{P}(\Omega)=\{S\subseteq \Omega \} $ be the power-set of $\Omega$. For $k\geq 1$ independent walks with transition matrix $\mathbf{P}$ on the same state space $\Omega$ we define the \textit{$k$-walk auxiliary chain} $\mathbf{Q}:=\mathbf{Q}(\mathbf{P},k)$ to be the Markov Chain on state space $V:=V(\Omega,k)$ given by \[V=\big\{((x_1, \dots, x_k) , S) : S\subseteq \Omega,\; x_i\in S\text{ for all }1\leq  i\leq k \big\}\subseteq \Omega^k\times \mathcal{P}(\Omega),\] with transition matrix specified by  \begin{equation}\label{eq:multitransM}\mathbf{Q}\left((\mathbf{x},S),(\mathbf{y},S\cup\{y^{(1)}, \dots, y^{(k)}\})\right)= \mathbf{P}(x^{(1)},y^{(1)})\cdots \mathbf{P}(x^{(k)},y^{(k)}), \end{equation} for any $S\subseteq \Omega$ and  $\mathbf{x},\mathbf{y}\in \Omega^k$ where $\mathbf{x}=(x^{(1)}, \dots, x^{(k)})$ and $\mathbf{y}=(y^{(1)}, \dots, y^{(k)})$.

Figure \ref{fig:aux} shows an example of the auxiliary chain $\mathbf{Q}$ of a single Markov chain $\mathbf{P}$ on three states, that is the case $k=1$. Staying within the confines of $k=1$ case for simplicity, one may think of $\mathbf{Q}$ as inducing a directed graph consisting of many `layers', where each layer is a copy of $\mathbf{P}$ restricted to a subset of $ \Omega$. These layers are linked by directed edges which are crossed when a new state not in the current layer is first visited. Thus, since a sequence $x_0,x_1, \dots$ in the first component of $V$ evolves according to $\mathbf{P}$ by \eqref{eq:multitransM}, each layer encodes which states of chain have been visited so far by a trajectory in $\mathbf{P}$.

\begin{figure}
	
\begin{center}
	
	\begin{tikzpicture}[label/.style={thick,circle}]
	\usetikzlibrary{arrows.meta}
	\usetikzlibrary{decorations.markings}
	\usetikzlibrary{decorations.pathreplacing}
	\tikzset{->-/.style={decoration={
				markings,
				mark=at position .5 with {Stealth[length=4mm]}},postaction={decorate}},>=stealth'}
	
	
	\draw (0 ,0) node {{$\color{blue}(1,\{1\})$}};
	\draw[->] (.6,0 ) -- (1.75 ,0) ;
	\draw (1.2 ,0) node[anchor=north]{{\scriptsize$1$}};
	
	\draw (0 ,3) node {{$\color{blue}(2,\{2\})$}};
	\draw[->] (.6,3 ) -- (1.75 ,4) ;
	\draw (1.2 ,3.4) node[anchor=north]{{\scriptsize$1/2$}};
	\draw[->] (.6,3 ) -- (1.75 ,2) ;
	\draw (1.2 ,2.4) node[anchor=north]{{\scriptsize$1/2$}};
	
	\draw (0 ,7) node {{$\color{blue}(3,\{3\})$}};
	\draw[->] (.6,7 ) -- (1.75 ,8) ;
	\draw (1.2 ,7.4) node[anchor=north]{{\scriptsize$1/2$}};
	\draw[->] (.6,7 ) -- (1.75 ,6) ;
	\draw (1.2 ,6.4) node[anchor=north]{{\scriptsize$1/2$}};

	\draw (2.5 ,0) node {{$\color{blue}(2,\{1,2\})$}};
	\draw[->] (2.45 ,.25) to[out=110,in=250 ] (2.45 ,1.75) ;
	\draw (2 ,1.2) node[anchor=north]{{\scriptsize$1/2$}};
	\draw[->] (2.55 ,1.75) to[out=290,in=70 ]  (2.55 ,.25) ;
	\draw (3 ,1.2) node[anchor=north]{{\scriptsize$1$}};
	\draw[->] (3.25 ,0) -- (6.1 ,2.9) ;
	\draw (4.5 ,2) node[anchor=north]{{\scriptsize$1/2$}};

	\draw (2.5 ,2) node {{$\color{blue}(1,\{1,2\})$}};
	
	\draw (2.5 ,4) node {{$\color{blue}(3,\{2,3\})$}};
	\draw[->] (2.45 ,4.25) to[out=110,in=250 ] (2.45 ,5.75) ;
	\draw (2 ,5.2) node[anchor=north]{{\scriptsize$1/2$}};
	\draw[->] (2.55 ,5.75) to[out=290,in=70 ]  (2.55 ,4.25) ;
	\draw (3 ,5.2) node[anchor=north]{{\scriptsize$1/2$}};
	\draw[->] (3.25,4.1 ) -- (4.6 ,4.9) ;
	\draw (4 ,4.5) node[anchor=north]{{\scriptsize$1/2$}}; 
	\draw[->] (3.25,5.9 ) -- (4.6 ,5.1) ;
	\draw (4 ,6 ) node[anchor=north]{{\scriptsize$1/2$}};
	
	\draw (2.5 ,6) node {{$\color{blue}(2,\{2,3\})$}};
	
	\draw (2.5 ,8) node {{$\color{blue}(1,\{1,3\})$}};
	\draw[->] (3.25 ,8) -- (6.1 ,7.1) ;
	\draw (4.5 ,7.6) node[anchor=north]{{\scriptsize$1$}}; 
	
	\draw (7 ,3) node {{$\color{blue}(3,\{1,2,3\})$}};
	\draw[->] (7.7 ,3.3) to[out=80,in=280 ] (7.7 ,6.7) ;
	\draw (8.2 ,5.3) node[anchor=north]{{\scriptsize$1/2$}}; 
	\draw[->] (7.5 ,6.7)  to[out=260,in=100 ] (7.5 ,3.3) ;
	\draw (7.05 ,5.3) node[anchor=north]{{\scriptsize$1/2$}}; 
	\draw[->] (6.5 ,3.3) -- (5.5 ,4.7) ;
	\draw (5.8 ,4) node[anchor=north]{{\scriptsize$1/2$}}; 
	
	\draw (5.5 ,5) node {{$\color{blue}(1,\{1,2,3\})$}};
	
	\draw (7 ,7) node {{$\color{blue}(2,\{1,2,3\})$}};
	\draw[->] (6.5 ,6.7) to[out=220,in=70 ] (5.5 ,5.3) ;
	\draw (5.7 ,6.5) node[anchor=north]{{\scriptsize$1/2$}}; 
	\draw[->] (5.7 ,5.3) to[out=50,in=240 ](6.7 ,6.7)  ;
	\draw (6.3 ,6) node[anchor=north]{{\scriptsize$1$}};
	\draw (6.7 ,4.8) node[anchor=north]{{\huge\color{red} $C$}};
	\draw[fill=red!20,opacity=0.4] (6.62,5) ellipse (2.05cm and 2.8cm);

	\draw[fill] (8,0) circle (.1);
	\draw (8 ,-.1) node[anchor=north]{{ $1$}};
	\draw[fill] (11,0) circle (.1);
	\draw (11 ,-.1) node[anchor=north]{{$3$}};
	\draw[fill] (9.5,2.6) circle (.1);
	\draw (9.5 ,2.7) node[anchor=south]{{$2$}};
	
	\draw[->] (8 ,0) to[out=85,in=215 ]  (9.4,2.54);
	\draw (8.2 ,2) node[anchor=north]{{\scriptsize$1$}};
	\draw[->] (9.5 ,2.6) to[out=265,in=35 ]  (8.08,.09);
	\draw (8.85,1.6) node[anchor=north]{{\scriptsize$  1/2$}};
	
	\draw[->] (11 ,0) to[out=95,in=325 ]  (9.6,2.54);
	\draw (10.8 ,2) node[anchor=north]{{\scriptsize$1/2$}};
	\draw[->] (9.5 ,2.6) to[out=275,in=155 ]  (10.92,.09);
	\draw (10.12,1.6) node[anchor=north]{{\scriptsize$  1/2$}};
	
	\draw[->] (11,0.0) -- (8.13 ,0) ;
	\draw (9.5 ,0) node[anchor=north]{{\scriptsize$1/2$}};
	\end{tikzpicture}
\end{center}
	\caption{This figure shows an example of a Markov chain $\mathbf{P}$ on three states (bottom right) and its associated auxiliary chain $\mathbf{Q}(\mathbf{P},1)$, where the set $C$ from Lemma \ref{multicovashit} is shown in the red shaded ellipse.}\label{fig:aux}
\end{figure}
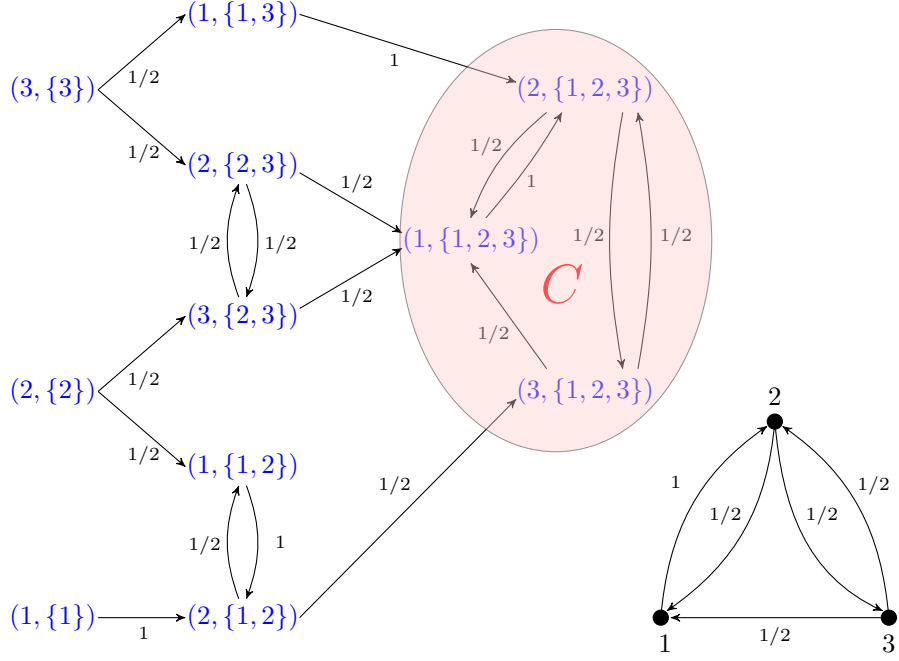

Similar constructions to $\mathbf{Q}(\mathbf{P},1)$ were used by the author and co-authors in the study of the Choice and $\varepsilon$-TB random walks, which are walks where a controller can influence which vertices are visited. In particular they were used to show that there exist optimal strategies for covering a graph by these walks which are time invariant in a certain sense \cite{POTC} and to show the computational problem of finding optimal strategies to cover a graph by these walks is in $\mathsf{PSPACE}$ \cite{TBRW}.

The next result equates the cover time by $k\geq 1$ multiple Markov chain with transition matrix $\mathbf{P}$ to the hitting time of a specific set in the auxiliary chain $\mathbf{Q}(\mathbf{P},k)$. For clarity we use the notation $\Exuc{\cdot}{\mathbf{P}}{\cdot} $ to highlight the chain, in this case $\mathbf{P}$, in which the expectation is taken.

\begin{lemma}\label{multicovashit}Let $\mathbf{P}$ be a Markov chain on $\Omega$, and let $k\geq 1$ be an integer. Let $\mathbf{Q}:=\mathbf{Q}(\mathbf{P},k)$ be the associated $k$-walk auxiliary chain with state space $V:=V(\Omega,k)$, and set $C = \{( \mathbf{u},\Omega) : \mathbf{u} \in \Omega^k \}\subset W .$ Then, for any $\mathbf{x}=(x^{(1)}, \dots, x^{(k)})\in \Omega^k$ and real number $a$, we have \[\Pruc{\mathbf{x}}{\mathbf{P}}{\tau_{\mathsf{cov}}^{(k)}\geq a}  = \Pruc{(\mathbf{x},\{x^{(1)}, \dots, x^{(k)}\})}{\mathbf{Q}}{\tau_{C}\geq a }.\] Consequently, $\Exuc{\mathbf{x}}{\mathbf{P}}{\tau_{\mathsf{cov}}^{(k)}}  = \Exuc{(\mathbf{x},\{x^{(1)}, \dots, x^{(k)}\})}{\mathbf{Q}}{\tau_{C}}$, for any $\mathbf{x} \in \Omega^k$.
\end{lemma}
We must introduce some notation before proving Lemma \ref{multicovashit}. For real valued random variables $X,Y$ we say that $Y$ \textit{stochastically dominates} $X$ if $ \Pr{ Y \geq a } \geq   \Pr{ X \geq a}$ for all real $a$, and we denote this by $X\preceq Y$. Thus, if $X\preceq Y$ and $Y\preceq X$, then $X$ and $Y$ are equidistributed. 

\begin{proof}[Proof of Lemma \ref{multicovashit}]
We first show how any trajectory $(\mathbf{X}_t)_{t\geq 0}$ of a $k$-multiple of the Markov chain $\mathbf{P}$ can be coupled with a trajectory $(Y_{t})_{t\geq 0}$ of the auxiliary Markov chain $\mathbf{Q}(\mathbf{P},k) $ given by \eqref{eq:multitransM}. To begin, given any start vector $\mathbf{X}_0 =\mathbf{x}_0  \in \Omega^k$, where $\mathbf{x}_0 =(x_0^{(1)}, \dots, x_0^{(k)} )$, we set $Y_0= (\mathbf{x}_0,\{x_0^{(1)}, \dots, x_0^{(k)} \}) \in V$. Then, given a trajectory $(\mathbf{X}_{t})_{t=0}^T= ( \mathbf{x}_t)_{t=0}^T$ we set $Y_{t}= \left(\mathbf{x}_t, \bigcup_{i=0}^t\bigcup_{j=1}^k\left\{x_i^{(j)} \right\}\right) $ for each $0\leq t\leq  T$. Now by \eqref{eq:multitransM}, 
	\begin{equation}\label{eq:trajprob2}\begin{aligned}\prod_{i=0}^{t-1}\prod_{j=1}^{k}\mathbf{P}(x_{i}^{(j)},x_{i+1}^{(j)}) &= \mathbf{Q}\left(\left(\mathbf{x}_0,\bigcup_{j=1}^{k}\left\{x_0^{(j)}\right\}\right),\left(\mathbf{x}_1,\bigcup_{i=0}^1\bigcup_{j=1}^{k}\left\{x_i^{(j)}\right\}\right)\right) \cdots\\ &\qquad\cdot\mathbf{Q}\left(\left(\mathbf{x}_{t-1},\bigcup_{i=0}^{t-1}\bigcup_{j=1}^{k}\left\{x_i^{(j)}\right\}\right),\left(\mathbf{x}_{t},\bigcup_{i=0}^{t}\bigcup_{j=1}^{k}\left\{x_i^{(j)}\right\}\right)\right). \end{aligned}\end{equation}
	Thus given any trajectory $(\mathbf{X}_t)_{t\geq 0}$ of $\mathbf{P}$ we can find a trajectory $(Y_t)_{t\geq 0}$ of $\mathbf{Q}(\mathbf{P},k)$ with the same measure. To couple a given trajectory $(Y_t)_{t\geq 0}$ of $\mathbf{Q}(\mathbf{P},k)$ to a trajectory $(\mathbf{X}_t)_{t\geq 0}$ of $\mathbf{P}$ is even simpler; given $\textbf{Y}_t=(\textbf{y}_t, S_t)$ we simply `forget' the second component of $\textbf{Y}_t$ and set $\textbf{X}_t= \textbf{y}_t \in \Omega^k$ for each $t\geq 0$. Again the measure is preserved by \eqref{eq:trajprob2}. 
	
	Recall the set $C = \{( \mathbf{u},\Omega)\, : \,\mathbf{u} \in \Omega^k \}$ which is a subset of the state space $V(\Omega, k)$ of the auxiliary chain $\mathbf{Q}(\mathbf{P}, k)$. To complete the proof we show that, for any $\mathbf{x} \in \Omega^k$, the times $\tau_{\mathsf{cov}}$ and $\tau_{C}$  in the coupled chains $\mathbf{X}_t$ and $\mathbf{Y}_t$, started from $\mathbf{x} $ and $(\mathbf{x},\{x^{1},\dots, x^{(k)} \})$ respectively, are equidistributed.

	Suppose we take any trajectory $(\mathbf{X}_{t})_{t=0}^{T}$ of length $T\geq 0$ such that $\cup_{i=0}^T\cup_{j=1}^k\{X_i^{(j)} \} =\Omega$. Then by the coupling above, we have  $\mathbf{Y}_T = \left(\mathbf{x}_T,\cup_{i=0}^T\cup_{j=1}^k\{X_i^{(j)} \}\right) = \left(\mathbf{x}_T,\Omega \right)\in C$. Since this holds for any trajectory and any time $T$ such that $\cup_{j=0}^T\{X_j \} =\Omega$, we can assume that $T$ is the first such time. That is, we can take $T=\tau_{\mathsf{cov}}$ and then it follows that $ \tau_C \preceq \tau_{\mathsf{cov}}$. 
	
	Conversely, let $(Y_t)_{t= 0}^T $ be any trajectory in $\mathbf{Q}$ where $Y_0= (\mathbf{y}_0,\cup_{j=1}^k\{y_0^{(j)}\})$, for some $\mathbf{y}_0 = (y_0^{(1)}, \dots,y_0^{(k)} ) \in \Omega^k$ and $Y_T \in C$. Since the only transitions supported by $\mathbf{Q}$ are from $(\mathbf{y},S )$ to $\left(\mathbf{z},S\cup\left(\cup_{j=1}^k \{z^{(j)}\} \right)\right)$ where $\prod_{j=1}^k\mathbf{P}(y^{(j)},z^{(j)})>0$, and $\mathbf{Y}_0= (\mathbf{y}_0,\cup_{j=1}^k \{y_0^{(j)}\})$, it follows that $\cup_{t=0}^{T}\cup_{j=1}^k \{y_t^{(j)}\} =\Omega$. Thus, by the coupling above, $\cup_{t=0}^{T}\cup_{j=1}^k \{ X_t^{(j)}\} =\Omega$. Similarly, since we can take $T=\tau_{C}$ to be minimal, we have $ \tau_{\mathsf{cov}}\preceq \tau_C$. 
	
	Thus for any pair of coupled trajectories with fixed start vertices $\mathbf{x} $ and $(\mathbf{x},\cup_{j=1}^k \{x_0^{(j)}\})$ the times $\tau_{\mathsf{cov}}$ and $\tau_{C}$ are the same. The final statement then follows by taking expectation.\end{proof}
Lemma \ref{multicovashit} equates the cover time of any Markov chain $\mathbf{P}$ on $\Omega$ (not just rational chains) to a hitting time in a higher dimensional chain $\mathbf{Q}$ on $V$. This result may be useful for studying the cover time of of an arbitrary Markov chain $\mathbf{P}$ on $\Omega$. However, one drawback of this approach is that for many chains $|V|$ is exponential in $|\Omega|$.

Having established Lemmas \ref{hitrat+} and \ref{multicovashit} the proof of Theorem \ref{multicovrat} is simple.

\begin{proof}[Proof of Theorem \ref{multicovrat}] Let $\mathbf{Q}:=\mathbf{Q}(\mathbf{P},k)$ be the auxiliary chain associated with the $k$-multiple Markov chain with transition matrix  $\mathbf{P}$. Then $\Exuc{\mathbf{x}}{\mathbf{P}}{\tau_{\mathsf{cov}}^{(k)}}  = \Exuc{(\mathbf{x},\{x^{(1)}, \dots, x^{(k)}\})}{\mathbf{Q}}{\tau_{C}}$ for any $\mathbf{x}=(x^{(1)}, \dots, x^{(k)})\in\Omega^k$ by Lemma \ref{multicovashit}, where $C=\{(\mathbf{y},\Omega) \mid \mathbf{y}\in \Omega^k \}$. By assumption we have $\Exuc{\mathbf{x}}{\mathbf{P}}{\tau_{\mathsf{cov}}}<\infty$ and so $\mathbf{x}\in B(C)$. It follows from Proposition \ref{hitrat+} that $\Exuc{(\mathbf{x},\{x^{(1)}, \dots, x^{(k)}\})}{\mathbf{Q}}{\tau_{C}} \in \mathbb{Q}$ and so $\Exuc{\mathbf{x}}{\mathbf{P}}{\tau_{\mathsf{cov}}} \in \mathbb{Q}$ as claimed.  
\end{proof}

\section*{Acknowledgements}
We thank Parsiad Azimzadeh and Agelos Georgakopoulos for discussions which lead to a simplification of the proof of Proposition \ref{hitrat+}. The author is currently supported by Engineering and Physical Sciences  Research Council (ESPRC) grant number EP/T004878/1. This work was started while the author was supported by ERC Starting Grant no.\ 679660 at the University of Cambridge.

\bibliographystyle{plainurl}

\end{document}